\documentclass[preprint,12pt]{elsarticle}

\usepackage{amssymb}
\usepackage{amsmath}
\usepackage{amsthm}

\newtheorem{theorem}{Theorem}

\newtheorem{definition}{Definition}

\journal{Nuclear Physics B}

\begin{document}

\begin{frontmatter}



\title{A Geometry-Aware Operator Learning Framework for Interface Problems on Varying Domains}


\author[1]{Shanshan Xiao}
\author[2]{Ye Li} 
\author[1]{Zhongyi Huang\corref{cor1}}
\author[1]{Hao Wu}

\cortext[cor1]{Corresponding author. E-mail: zhongyih@tsinghua.edu.cn}

\affiliation[1]{organization={Dept. of Mathematical Sciences, Tsinghua University},
            city={Beijing},
            postcode={100084}, 
            country={China}}
\affiliation[2]{organization={College of Computer Science and Technology, Nanjing University of Aeronautics and Astronautics},
	city={Nanjing},
	postcode={211106}, 
	country={China}}

\begin{abstract}
	Solving Partial Differential Equation (PDE) interface problems on varying domains is a critical task in design and optimization, yet it remains computationally prohibitive for traditional solvers. Although operator learning has shown promise on fixed geometries, its potential for geometry-dependent interface problems has been largely unexplored. To bridge this gap, we propose an extension-based neural operator framework applicable to general linear interface problems. A key innovation of our method is the integration of the Tailored Finite Point Method (TFPM) with our base network, which reduces memory consumption and effectively alleviates the curse of dimensionality.  On the theoretical front, we establish the continuity of the Helmholtz operator with respect to domain perturbations and provide rigorous error estimates for the proposed encodings. Comprehensive numerical experiments demonstrate that our framework achieves state-of-the-art accuracy and robustness. Consequently, this work provides a powerful, data-efficient tool for varying-domain simulations, offering new possibilities for real-time shape optimization.
\end{abstract}




\end{frontmatter}



\section{Introduction}

Linear interface problems are ubiquitous in science and engineering. They play a pivotal role in modeling complex physical systems where distinct materials or phases interact, such as in heat conduction through composite materials, multi-phase flow dynamics, and electromagnetic field propagation in heterogeneous media\cite{wang2002computational, hashin1990thermoelastic, hubatsch2025transport, weber2004modeling}. These problems are mathematically characterized by the existence of coupling conditions between solutions across different interfaces, along with discontinuous coefficients or source terms at possible interface boundaries. This leads to non-smooth or jump solutions, posing significant challenges for numerical simulations.

To address this complex problem, researchers have proposed numerous numerical algorithms. The finite element method (FEM)\cite{babuvska1974solution, barrett1987fitted, chen1998finite} handles interfaces naturally by aligning meshes with the interface, though this often requires mesh updates for moving interfaces. To circumvent mesh restrictions, embedded or immersed methods such as the immersed boundary method (IBM) and the immersed interface method (IIM) have been developed\cite{peskin2002immersed, leveque1994immersed}; these methods modify the numerical stencil or add forcing terms to account for solution jumps on a fixed Cartesian grid. The discontinuous Galerkin (DG) method\cite{massjung2012unfitted, mu2016new, saye2019efficient} offers high-order accuracy and flexibility in handling discontinuities by allowing basis functions to be discontinuous across elements. Finite difference methods\cite{huang2009tailored}, often combined with interface-tracking techniques, provide simplicity and efficiency, particularly when adapted with correction terms to maintain accuracy across interfaces. More recent approaches, such as the Voronoi interface method (VIM)\cite{guittet2015solving, helgadottir2018solving}, leverage computational geometry to discretize irregular domains with robust accuracy. Together, these methods form a comprehensive toolkit for solving interface problems, each offering distinct advantages in terms of accuracy, geometric flexibility, and computational efficiency.

Building upon the success of neural networks in computational mathematics, researchers have recently extended their application to solving interface problems for PDEs. For example, Wang\cite{wang2020mesh} combined a shallow neural network for boundary condition approximation with the Deep Ritz Method to address variational formulations of interface problems, and He\cite{he2022mesh} use one NN structure to capture the non-smooth or even discontinuous solutions. Similarly, Liu\cite{liu2020multi} developed a multi-scale neural network to solve the Poisson-Boltzmann equation using variational principles. A prominent direction in this area is the fusion of domain decomposition with deep learning. Dwivedi’s Distributed PINN (DPINN)\cite{dwivedi2019distributed} method partitions the domain into disjoint subregions, each handled by a separate subnet trained via a composite loss function. He\cite{he2022mesh} extended this idea by introducing static weights into the loss and an adaptive sampling scheme to boost accuracy. Jagtap and colleagues further advanced the field with XPINNs\cite{jagtap2020extended}, which apply domain decomposition to solve nonlinear PDEs in complex domains. These developments have inspired subsequent methods such as DeepDDM\cite{li2020deep} and Interface Neural Networks (INN)\cite{wu2022inn}, reinforcing the integration of deep learning and domain decomposition for interface problems.

But in some applications, such as geometric optimization and inverse material identification\cite{edelman1998geometry, absil2008optimization}, we frequently encounter scenarios where the computational domain, interface location, or physical parameters vary continuously. For these problems, relying on traditional numerical methods (e.g., FEM or FDM) is computationally prohibitive, as they necessitate mesh regeneration and iterative solving for every new configuration. Similarly, conventional neural network approaches are often restricted to a fixed setting; any change in the geometry or parameters typically requires retraining the network from scratch, which incurs significant time and computational costs.

To overcome these bottlenecks, Operator Learning has emerged as a transformative paradigm\cite{lu2021learning, lifourier}. Currently, there have been some research advancements in operator learning for interface problems on fixed domains\cite{wu2024solving, li2024tailored, du2024physics}. Building upon these, we aim to explore interface problems on varying domains. In existing studies, approaches to handling varying domains primarily fall into two categories: extension-based methods and deformation-based methods. Among them, deformation-based methods aim to map irregular domains to a reference domain, thereby establishing a mapping to the Banach space on the reference domain. Based on this, existing operator learning architectures (such as DeepONet, FNO, etc.) can be applied. Representative deformation-based approaches include Geo-FNO, D2D \& D2E, and DIMON\cite{li2023fourier, xiao2024deformation, yin2024dimon}. Another approach is extension-based methods, which seek to cover all varying domains within a larger fixed domain and extend functions defined on different domains to this larger domain. This transforms functions originally defined on different domains into elements of the Banach space on the extended domain. For example, GINO is an extension-based method\cite{li2023geometry}.

Building upon this, we consider the more complex varying-domain interface problems, addressing cases where both the domain and possibly multiple interfaces undergo changes. Given the complexity of simultaneous variations in the domain and interfaces, as well as the presence of multiple interfaces, we adopt an extension-based approach to handle changes in both the domain and interfaces. For the following general linear interface PDE:
\begin{equation}
	\begin{cases}
		\mathcal{L}u = f, & \text{in } \Omega \backslash \Gamma, \\
		\mathcal{B}u = g, & \text{on } \partial \Omega, \\
		\mathcal{F}_{1}(u^{+},u^{-},\frac{\partial u^+}{\partial n}, \frac{\partial u^-}{\partial n}) = 0, & \text{on } \Gamma, \\
		\mathcal{F}_{2}(u^{+},u^{-},\frac{\partial u^+}{\partial n}, \frac{\partial u^-}{\partial n}) = 0, & \text{on } \Gamma,
	\end{cases}
\end{equation}
we first transformed the complex solution mapping into a mapping from a Banach space to a Banach space via zero extension, enabling the use of neural operators for learning. Subsequently, we proposed an encoding method based on a standard rectangular grid, proved the sensitivity of the Helmholtz equation to variations in the domain and interface, and provided a simple error estimate for encoding the domain using the characteristic function and the signed distance function (SDF). In addition, we improved the basic method by incorporating the Tailored Finite Point Method (TFPM)\cite{huang2009tailored, han2008tailored, han2009tailored, han2008tailored2} basis, which reduces the storage requirements during training and achieves good predictive performance while saving GPU memory. Also we conducted a series of experiments to demonstrate the accuracy of our proposed method.

The remainder of this paper is organized as follows. In Section 2, we provide the mathematical formulation of the interface problems on variable domains and introduce the necessary notations. Section 3 details our proposed methodology, describing the neural operator architecture, the specific encoding schemes, and the novel framework integrating the Tailored Finite Point Method. Section 4 is devoted to the theoretical analysis, where we establish the shape sensitivity of the Helmholtz equation with respect to domain and interface perturbations and derive error estimates for the domain encoding. In Section 5, we present extensive numerical experiments to validate the accuracy and robustness of the proposed method. Finally, concluding remarks are given in Section 6.

\section{Problem setup}
First of all, we consider the interface problem as following:

\begin{equation}
	\begin{cases}
		\mathcal{L}u = f, & \text{in } \Omega \backslash \Gamma, \\
		\mathcal{B}u = g, & \text{on } \partial \Omega, \\
		\mathcal{F}_{1}(u^{+},u^{-},\frac{\partial u^+}{\partial n}, \frac{\partial u^-}{\partial n}) = 0, & \text{on } \Gamma, \\
		\mathcal{F}_{2}(u^{+},u^{-},\frac{\partial u^+}{\partial n}, \frac{\partial u^-}{\partial n}) = 0, & \text{on } \Gamma,
	\end{cases}
\end{equation}
Here, $\mathcal{L}$ is a linear PDE operator, $\mathcal{B}$ is a linear boundary operator, and $\mathcal{F}_{1}$ and $\mathcal{F}_{2}$ are linear interface condition operators. And we assume that all domains \(\Omega\) under consideration are bounded and that all interfaces $\Gamma$ are curves lying in the interior of $\overline{\Omega}$.

To better handle the abstract operator $\mathcal{L}, \mathcal{B}, \mathcal{F}_{1}, \mathcal{F}_{2} $ introduced above, we can consider a case of the most common Helmholtz equation interface problem as follows:
\begin{equation}
	\begin{split}
		\begin{cases}
			-\nabla(a(x)\nabla u(x)) + b(x)u(x) = f(x), \quad \text{for}\quad x \in \Omega \backslash \Gamma , \\
			u|_{\partial \Omega} = g(x), \\
			[u]|_{\Gamma} = \alpha, \\
			[a\nabla u \cdot \textbf{n}]|_{\Gamma} = \beta,
			\label{interfacepde}
		\end{cases}
	\end{split}
\end{equation}
and $\textbf{n}$ denotes the unit normal vector pointing from the internal toward the external of the interface. We aim to learn the solution mapping of \eqref{interfacepde} on varying domains, contains both the outer boundary $\Omega$ and the interface $\Gamma$, i.e., the mapping
\begin{equation}
	\mathcal{G}: a(x)_{\Omega} \times b(x)_{\Omega}  \times f(x)_{\Omega} \times g(x)_{\partial \Omega} \times \alpha_{\Gamma} \times \beta_{\Gamma} \rightarrow u(x).
\end{equation}

After introducing the problem to be solved, we will propose a general method for learning the mapping $\mathcal{G}$.

\section{Methods}
For varying-domain problems, two principal methodologies are prevalent: the extension method and the deformation method. In light of the complexity arising from the interface conditions and the variation of the interface $\Gamma$, we favor the extension-based approach in the present study. Inspired by GINO, we employ the aforementioned specialized encoding approach and adopt an FNO-based network architecture to learn $\mathcal{G}$.

\subsection{Neural network architecture}

Back to the original problem, we aim to learn the mapping
\begin{equation}
	\mathcal{G}: a(x)_{\Omega} \times b(x)_{\Omega}  \times f(x)_{\Omega} \times g(x)_{\partial \Omega} \times \alpha_{\Gamma} \times \beta_{\Gamma} \rightarrow u(x).
\end{equation}

However, here $\alpha_{\Gamma}, \beta_{\Gamma}$ and $a(x)_{\Omega}, b(x)_{\Omega}, f(x)_{\Omega}, g(x)_{\partial\Omega}$ do not constitute Banach space. Therefore, we first perform an extension of these functions, based on the assumption that all $\Omega$ are bounded domain in $\mathbb{R}^{d}$, and $\Gamma \subset \bar{\Omega}$, we can find a large square $[s, t]^d$ that covers all $\Omega$.

Based on this, taking $a(x)_{\Omega}$ and $\alpha_{\Gamma}$ as example we perform zero extension on these functions to $[s, t]^d$ as following:
\begin{equation}
	\hat{a}(x) = \begin{cases}
		a(x) & \text{if} \quad x \in \Omega, \\
		0 & \text{if} \quad x \notin \Omega,
	\end{cases}
	\quad \hat{\alpha}(x) = \begin{cases}
		\alpha(x) & \text{if} \quad x \in \Gamma, \\
		0 & \text{if} \quad x \notin \Gamma.
	\end{cases}
\end{equation}
Similarily, $u(x)_{\Omega}, b(x)_{\Omega}, f(x)_{\Omega}, g(x)_{\partial\Omega}, \beta_{\Gamma}$ can be extended as the same method. Thus, we obtain the extended function, denoted as $\hat{\alpha}, \hat{\beta}$ and $\hat{a}(x),$ $\hat{b}(x),$ $\hat{f}(x),$ $\hat{u}(x),$ $\hat{g}(x)$. Indeed, the extended functions lose information about the original domain. To address this, we additionally extend the domain $\Omega$ and the interface $\Gamma$ by representing them as functions on $[s, t]^{d}$, we denote them as $\phi_{\Omega}$ and $\phi_{\Gamma}$. Here, there are many choices for the extension method; we employ two of them based on the characteristic function and the signed distance function respectively. 

Based on this, we define the extended operator $\hat{\mathcal{G}}$:
\begin{equation}
	\hat{\mathcal{G}}: \phi_{\Omega} \times \phi_{\Gamma} \times \hat{a}(x) \times \hat{b}(x)  \times \hat{f}(x) \times \hat{g}(x) \times \hat{\alpha} \times \hat{\beta} \rightarrow \hat{u}(x).
\end{equation}
Note that here $\hat{\mathcal{G}}$ is an operator from tensor of several Banach space to one Banach space, and it satisfies
$$\mathcal{G}(a_{\Omega},b_{\Omega},f_{\Omega},g_{\Omega},\alpha_{\Gamma}, \beta_{\Gamma}) = \hat{\mathcal{G}}(\phi_{\Omega} ,\phi_{\Gamma},\hat{a} ,\hat{b}, \hat{f},\hat{g} ,\hat{\alpha} ,\hat{\beta})|_{\Omega}.$$

Below, we provide a detailed definition of the extension operator.

\begin{definition}
	Suppose $\Omega$ is a bounded domain in $\mathbb{R}^{n}$, $\Gamma$ is an interface within $\Omega$, and $a,b,f$ are piecewise continuous functions on $\Omega$, while $g,\alpha,\beta$ are piecewise continuous functions on $\partial\Omega$ and $\Gamma$, respectively. Then, for the mapping $$\mathcal{G}: a(x)_{\Omega} \times b(x)_{\Omega}  \times f(x)_{\Omega} \times g(x)_{\partial \Omega} \times \alpha_{\Gamma} \times \beta_{\Gamma} \rightarrow u(x),$$
	where $u(x)$ is the solution to the equation \eqref{interfacepde}, we have the extended operator $\hat{\mathcal{G}}$ as defined above. The relationship between $\mathcal{G}$ and $\hat{\mathcal{G}}$ satisfies: 
	\begin{equation}
		\mathcal{G}(a_{\Omega},b_{\Omega},f_{\Omega},g_{\Omega},\alpha_{\Gamma}, \beta_{\Gamma}) = \hat{\mathcal{G}}(\phi_{\Omega} ,\phi_{\Gamma},\hat{a} ,\hat{b}, \hat{f},\hat{g} ,\hat{\alpha} ,\hat{\beta})|_{\Omega}. 
		\label{operatorG}
	\end{equation}
	\label{defG}
\end{definition}

After defining the extension operator, we will next demonstrate that this extended operator is unique.

\begin{theorem}[Uniqueness]
	The operator $\hat{\mathcal{G}}$ defined by $\mathcal{G}$ in Definition \ref{defG} is unique, i.e., if there exists another operator $\hat{\mathcal{G}}_2$ satisfying equation \eqref{operatorG}, then $\hat{\mathcal{G}} = \hat{\mathcal{G}}_2$.
\end{theorem}

Theorem 1 is clearly valid, so we omit its proof. Then after transforming the problem into learning the operator $\hat{\mathcal{G}}$, we can employ neural operators\cite{lifourier, lu2021learning} for learning. Therefore, in the next subsection, we will briefly introduce the architecture of the neural operator we use.

\subsection{Neural operator}

After obtaining $\hat{\mathcal{G}}$, we proceed to approximate it using a neural operator, which we denote as $\mathcal{G}_{\theta}$. To this end, we opt for the Fourier Neural Operator (FNO) as our specific approximant. And $\mathcal{G}_{\theta}$ can be represent as 
\begin{equation}
	\mathcal{G}_{\theta} = \mathcal{Q} \circ (W_{L} + \mathcal{K}_{L} + b_{L}) \circ \cdots \circ \sigma (W_{1} + \mathcal{K}_{1} + b_{1}) \circ \mathcal{P},
\end{equation}
here $\mathcal{P}: \mathbb{R}^{d_{input}} \rightarrow \mathbb{R}^{d_{1}}$ are pointwise neural networks that encode the input functions into higher dimensional space and  $\mathcal{Q}: \mathbb{R}^{d_{L}} \rightarrow \mathbb{R}^{d_{u}}$ the decode neural networks. And $\mathcal{K}_{L}:\{D\rightarrow\mathbb{R}^{d_{L+1}}\}$ are Fourier integral kernel operators, $b_{L}:D\rightarrow \mathbb{R}^{d_{L+1}}$ area the bias term of $L$-th layer, and $\sigma$ are activate functions.
\vspace{\baselineskip}\\
\textbf{Fourier integral operator}

The Fourier integral operator is defined as 
\begin{equation}
	(\mathcal{K}(\phi)v_{t})(x) = \mathcal{F}^{-1}(R_{\phi}\cdot (\mathcal{F}v_{t}))(x), \quad \forall x\in D,
	\label{fo}
\end{equation}
and here $R_{\phi}$ is the Fourier tranform of periodic function $\kappa: \Tilde{D}\rightarrow \mathbb{R}^{d_{v}\times d_{v}}$ parameterized by $\phi\in\Theta_{\mathcal{K}}$.

After introducing the network architecture, we will present the methods for encoding the input functions and the domain.

\subsection{Encoder domains and functions}

Under our assumptions, for a set of bounded domains $\{\Omega_{i}\in\mathbb{R}^{d}\}_{I}$, we can choose a big square domain $[s,t]^d$ that covers all $\Omega_{i}$, Then we obtain a uniform grid of size $n \times n \times \cdots \times n$ on $[s,t]^d$.  Next, we'll introduce how to encode our input.
\vspace{\baselineskip}\\
\textbf{Encode domain and interface}

First, we encode the domain and interface information. Without loss of generality, we assume that the domain $\Omega \subset [s,t]^d$ contains $m$ non-intersecting closed interfaces $\Gamma_i$. Following the extension assumption introduced earlier, we now have $m+1$ functions $\phi_{\Omega}, \phi_{\Gamma_i}$ defined on $[s,t]^d$, encoding the domain and each distinct interface, respectively. Next, depending on whether the characteristic function or the signed distance function (SDF) is used for the extension of domain information functions, we have two encoding approaches.

For the extension based on the characteristic function, we only need an $n^d$-dimensional matrix for encoding, and the encoding method is as follows:
\begin{equation}
	\Phi_{\mathbf{K}} = 
	\begin{cases}
		0, & p_{\mathbf{K}} \text{ in } [s,t]^d \backslash \Omega, \\
		1, & p_{\mathbf{K}} \text{ in } \Omega \backslash \cup \Gamma_{j}, j = 1,2,\cdots, m, \\
		i+1, & p_{\mathbf{K}} \text{ in } \text{int}(\Gamma_{i}) ,
	\end{cases}
\end{equation}
where $ p_{\mathbf{k}} $ denotes the grid point indexed by $\mathbf{K} =  (k_1, \cdots, k_d) $, and $ \text{int}(\Gamma_i) $ represents the bounded open set enclosed by the closed curve $ \Gamma_i $.

For the extension based on the signed distance function (SDF), our representation differs slightly. Although the SDF offers higher accuracy, we need to use a larger $ (m + 1) \times n^d $-dimensional matrix for encoding in order to preserve distance information. The encoding method is as follows:
\begin{equation}
	\Phi_{j, \mathbf{K}} = 
	\begin{cases}
		\phi_{\Omega}(p_{\mathbf{K}}), & j = 1, \\
		\phi_{\Gamma_{j-1}}(p_{\mathbf{K}}), & j>1, \\
	\end{cases} \quad j = 1,2,\cdots, m,
\end{equation}
where $ \Phi_{j, \mathbf{K}} $ represents the value at grid point $ p_{\mathbf{k}} $ in the $ j $-th $ n^d $-dimensional matrix, $\phi_{\Omega}, \phi_{\Gamma_{j}}$ are signed distance function of $\Omega, \Gamma_{j}$ respectively.

It is evident that there is a trade-off between the two encoding strategies. While the signed distance function (SDF) provides superior accuracy (we will later establish the relationship between the encoding error and the grid point density), it demands substantially more encoding capacity in scenarios involving multiple internal interfaces, in contrast to the characteristic function encoding which remains compact.
\vspace{\baselineskip}\\
\textbf{Encode functions defined on $\Omega$}

For functions defined on the entire domain $\Omega$, we have already performed zero extension, so the encoding can be simply defined as the values of the extended function at the grid points. For example, a function $ f(x) $ is encoded into an $ n^d $-dimensional matrix with:
\begin{equation}
	F_{\mathbf{K}} = \begin{cases}
		f(p_{\mathbf{K}}) & \text{if} \quad p_{\mathbf{K}} \in \Omega_{i}, \\
		0 & \text{if} \quad p_{\mathbf{K}} \notin \Omega_{i}.
	\end{cases}
\end{equation}
\vspace{\baselineskip}\\
\textbf{Encode functions defined on a low-dimensional manifold}

For the encoding of functions defined on low-dimensional manifolds—spec\-ifically, in our problem, the boundary conditions on $\partial \Omega$ and the interface conditions defined on $\Gamma_i$—these functions are defined on lower-dimensional manifolds. Simply taking the extended function values at grid points would result in a loss of information. Therefore, we introduce two encoding methods different from those mentioned above.

First, we present a simple encoding approach based solely on the function values at the nearest grid points, consider the boundary $\partial \Omega$ (For functions defined on the internal interface \(\Gamma\), the encoding approach is completely analogous.), $ g $ is a piecewise continuous function defined on the boundary, we encode the function $ g $ as follows:
\begin{equation}
	G_{\mathbf{K}} = g(\mathcal{P}_{\partial \Omega}(p_{\mathbf{K}}))\cdot  \exp(-\text{sdf}_{\partial \Omega}(p_{\mathbf{K}})² / (2\sigma^{2})),
\end{equation}
here $\mathcal{P}_{\partial\Omega}(p_{\mathbf{K}})$ denotes the projection of point $p_{\mathbf{K}}$ onto $\partial \Omega$, and we multiply by a Gaussian weight. This Gaussian weight depends on the distance from the grid point to the low-dimensional manifold $\partial \Omega$: as the distance increases, the weight decays to zero, while on $\partial \Omega$ itself, the weight equals one. At the same time, in a simple case—for example, when $ g $ is a constant function—we may omit the Gaussian weighting part.

\subsection{Neural operator with TFPM basis}

As proposed earlier, $\hat{\mathcal{G}}(\cdot)$ outputs a function value $c = \{c_{ij}\}$ at each node of an $ n \times n $ grid. This function can be interpreted as the coefficients of the first-order finite element basis functions on the grid. Specifically, if we consider the finite element basis functions $\phi_{ij}$ defined over the grid, 
the finite element approximation of the equation’s solution on the grid can be expressed as $ u = \sum_{i,j} c_{ij} \phi_{ij}.$

On this basis, if we consider replacing $\phi_{ij}$ with other global or local basis functions, we may obtain better properties. For instance, global basis functions can simplify the structure of the neural operator, while the Tailored Finite Point Method (TFPM)-based basis functions—introduced later—can outperform finite element basis functions in certain aspects due to their excellent approximation capability within small domains. To introduce the TFPM basis, we first briefly outline the core idea of the TFPM.
\vspace{\baselineskip}\\
\textbf{Tailored Finite Point Method}

First, we briefly introduce the fundamental idea of the Tailored Finite Point Method (TFPM). Consider the Helmholtz equation interface problem as shown in Equation \eqref{interfacepde}, and perform the following variable substitution:
\begin{equation}
	\begin{split}
		&y(x)  = \int_{x}^{x}\frac{1}{a(\xi)} d\xi, \quad \text{for} \quad x \in \Omega \\
		&c(y) \equiv a(x(y))b(x(y)),\quad F(y)\equiv a(x(y))f(x(y)),
	\end{split}
\end{equation}
then, the transformed variable $U(y)\equiv u(x(y))$ satisfies the following PDE:
\begin{equation}
	\begin{split}
		&-U^{''}(y) + c(y)U(y) = F(y), \quad \text{for}\quad x \in \Omega \backslash \Gamma\\
		& u|_{\partial \Omega} = g(x), \\
		& [u]|_{\Gamma} = \alpha, \quad [au']|_{\Gamma} = \beta.
	\end{split}
	\label{transform}
\end{equation}
Thus, equation \eqref{transform} can be handled using TFPM. First, we partition the domain $\Omega$ into a collection of small subregions $\{\Omega_{i}\}$. Then, within each subregion $\Omega_{i}$, TFPM approximates the parameter $c(y), F(y)$ as a constant $c_{j} = \mu_{j}^{2}, F_{j}$, so that on each small subregion, the equation simplifies to:
\begin{equation}
	-\Delta u + \mu_{j}^{2} u = F_{j}, \quad x\in\Omega_{j}.
\end{equation}
Thus, on each $\Omega_{j}$, the solution to this equation can be expanded in the following form:
\begin{equation}
	u_{j} = F_{j}/\mu_{j}^{2} + c_{0j}e^{\mu_{j}x} + c_{1j}e^{-\mu_{j}x} + c_{2j}e^{\mu_{j}y} + c_{3j}e^{-\mu_{j}y}.
	\label{tfpm}
\end{equation}
Based on this, for each subdomain $\Omega_{j}$, the TFPM considers the connection conditions between the midpoints of the four edges of the small square and other subdomains (or boundaries/interfaces). This leads to a corresponding system of linear equations $Ax = B$, from which the values of the parameters $\{c_{0j}, c_{1j}, c_{2j}, c_{3j}\}$ are obtained. In addition, the local bases $\{e^{\mu_{j}x}, e^{-\mu_{j}x}, e^{\mu_{j}y}, e^{-\mu_{j}y}\}$ here are referred to as the TFPM basis in this paper.

\begin{figure}[t]
	\centering
	\includegraphics[width=1.0\textwidth]{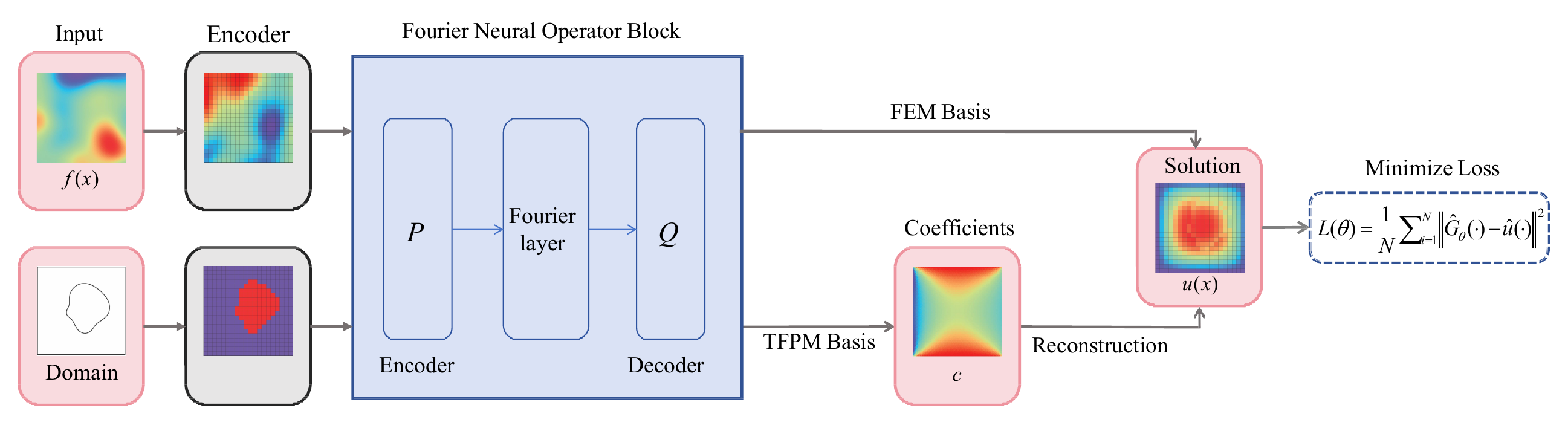}
	\caption{Neural network architecture}
	\label{network}
\end{figure}
\vspace{\baselineskip}
\textbf{Neural operator architecture with TFPM basis}

Given the characteristic of TFPM, which can represent the solution of the equation using parameterized formulations \eqref{tfpm} within subdomains, we propose a learning method that incorporates TFPM basis.

In this method, we learn the coeffients of \eqref{tfpm} rather than solution $u$, therefore the operator $\mathcal{G}$ becomes 
\begin{equation}
	\mathcal{G}_{TFPM}: a(x)_{\Omega} \times b(x)_{\Omega}  \times f(x)_{\Omega} \times g(x)_{\partial \Omega} \times \alpha_{\Gamma} \times \beta_{\Gamma} \rightarrow c,
\end{equation}
here $c$ is the coeffients of TFPM. 

Besides, we can also define a reconstruction mapping $\mathcal{R}e$ that maps coeffients $c$ to solution $u$ based on \eqref{tfpm}, then the map $\mathcal{R}e\circ \mathcal{G}_{TFPM}$ becomes:
\begin{equation}
	a(x)_{\Omega} \times b(x)_{\Omega}  \times f(x)_{\Omega} \times g(x)_{\partial \Omega} \times \alpha_{\Gamma} \times \beta_{\Gamma} \rightarrow u(x).
\end{equation}

Thus, we have derived a neural operator based on the TFPM basis, and we will later demonstrate that the TFPM basis can reduce the required GPU memory during training. Moreover, since this method reduces the size of the training dataset, it can alleviate the computational and storage burdens associated with increasing dimensionality to some extent.

\subsection{Loss Function}

For networks using two different basis functions, we compute the end-to-end loss in each case. Specifically, for the finite element basis, the loss function is 
$$L(\theta) = \frac{1}{N}\sum_{i = 1}^{N}\Vert \hat{G}_{\theta}(\cdot) - \hat{u}(\cdot) \Vert^{2}_{L^{2}([s,t]^{d})},$$
 and for the TFPM basis, the loss function is 
 $$L(\theta) = \frac{1}{N}\sum_{i = 1}^{N}\Vert \mathcal{R}e \circ \hat{G}_{\theta}(\cdot) - \hat{u}(\cdot) \Vert^{2}_{L^{2}([s,t]^{d})}.$$
 Finally, we present our network architecture in Figure \ref{network}.

\section{Theoretical analysis}

In this section, we primarily conduct a theoretical analysis of the previously proposed network architecture, encoding methods, and properties of the PDE. First, we will prove the continuity of the solution to the Helmholtz equation with respect to the interface and the domain as variables.

\subsection{Continuity of the Helmholtz equation operator}
Here, we consider the Helmholtz equation in the following specific form with $b\in\mathbb{C}$ and Im($b)>0$ :
\begin{equation}
	\begin{split}
		& \Delta u + b^{2}u(x) = f(x), \quad \text{for}\quad x \in \Omega \backslash \Gamma\\
		& u|_{\partial \Omega} = g(x), \\
		& [u]|_{\Gamma} = \alpha, \quad [u']|_{\Gamma} = \beta.
		\label{helmholtz_c}
	\end{split}
\end{equation}
Based on the previous continuity assumption for the operator $\mathcal{G}$, we need to prove that the aforementioned Helmholtz operator is continuous with respect to variations in both the functions and the domain. 

Consider the extended operator:
\begin{equation}
	\hat{\mathcal{G}}: \phi_{\Omega} \times \phi_{\Gamma} \times \hat{b}(x)  \times \hat{f}(x) \times \hat{g}(x) \times \hat{\alpha} \times \hat{\beta} \rightarrow \hat{u}(x),
\end{equation}
the continuity of the operator $\hat{\mathcal{G}}$ with respect to the input functions is straightforward. Here, we focus solely on proving that $\hat{\mathcal{G}}$ is continuous with respect to variations of domain $\Omega$ and the interface $\Gamma$, i.e., $\phi_{\Omega}$ and $\phi_{\Gamma}$.

\begin{theorem}[Continuity of interface under constant interface conditions]
	\label{conti_c}
	For the operator $\hat{\mathcal{G}}$ defined above, assume $\Omega$ is a bounded domain in $\mathbb{R}^{n}$(n = 2 or 3), moreover, all interfaces considered here are smooth closed hypersurfaces within $\Omega$, then given constants $\alpha$ and $\beta$, it follows that $\hat{\mathcal{G}}$ is continuous with respect to variation $\Gamma$. That is, for any prescribed $\Gamma$, $\Gamma_r$, assume that there exists of a transformation $\tau_{r}$ from $\Gamma$ to $\Gamma_{r}$:
	\begin{equation}
		\tau_{r}: x\mapsto x + r(x),
	\end{equation}
	then for arbitrary $\varepsilon > 0$, there exists $\delta > 0$ such that if $\Vert r\Vert_{\infty} < \delta$, then the following inequality holds:
	\begin{equation}
		\big\Vert \hat{\mathcal{G}}(\phi_{\Gamma}, \dots) - \hat{\mathcal{G}}(\phi_{\Gamma_r}, \dots) \big\Vert_{2} < \varepsilon,
	\end{equation}
	here, all other input functions remain unchanged. Since $\alpha$ and $\beta$ are constants, the interface conditions are given by $\hat{\alpha}_{\Gamma}, \hat{\alpha}_{\Gamma_r}$ and $\hat{\beta}_{\Gamma}, \hat{\beta}_{\Gamma_r}$, respectively.
\end{theorem}
\begin{proof}
	First, based on the transformation $\tau_{r}$, for any function $\psi$ defined on $\Gamma$ and $\psi_{r}$ defined on $\Gamma_{r}$, there exists a transformation $\tau_{r}$ that maps it to $\Gamma_{r}$, along with an inverse transformation, satisfying 
	\begin{equation}
		\tau_{r}(\psi_{r}(x)) = \psi_{r}(x + r(x)), \quad \tau_{r}^{-1}(\psi)(x + r(x)) = \psi(x).
	\end{equation}
	Since both $\alpha$ and $\beta$ are constants here, we have:
	$$\hat{\alpha}_{\Gamma_r} = \widehat{\alpha_{\Gamma}\circ \tau^{-1}_{r}},\quad \hat{\beta}_{\Gamma_r} = \widehat{\beta_{\Gamma}\circ \tau^{-1}_{r}}.$$
	
	Next, we express the solution to the interface problem as an integral function of fundamental solutions of Helmholtz equation $\Phi$, i.e.,
	\begin{equation}
		\begin{split}
			u(x) = \int_{\Omega} \Phi(x,y) f(y)dy + &\int_{\partial\Omega}(\frac{\partial\Phi(x,y)}{\partial n_{y}}g(y) - \Phi(x,y)\frac{\partial u(y)}{\partial n}) dS_{y} + \\
			&\int_{\Gamma}(\Phi(x,y)\beta - \frac{\Phi(x,y)}{\partial n_{y}}\alpha)dS_{y}.
		\end{split}
	\end{equation}
	The three integral terms here can be viewed as the source term, the boundary correction term, and the interface correction term, respectively. Therefore, we only need to focus on the interface correction term related to the interface $\Gamma$, denoted as $ u_{\Gamma} $. Therefore, we only need to prove
	$$\Vert u_{\Gamma} - u_{\Gamma_{r}} \Vert_{2} < \varepsilon.$$
	
	To prove this conclusion, we cite a theorem from the literature by Martin Costabel\cite{costabel2012shape}. First, in Example 2.4, it is shown that the fundamental solution $\Phi$ is pseudo-homogeneous of class $-1$. Based on this, according to Theorem 4.8 in the same reference, the mapping 
	\begin{equation}
		r \mapsto \mathcal{P}_{r}\tau_{r}^{-1}
	\end{equation}
	is infinitely Gâteaux differentiable, and thus continuous. Here, $\mathcal{P}_{r}$ are potential operators defined by 
	\begin{equation}
		(\mathcal{P}_{r}u_{r})(x) = \int_{\Gamma_{r}} k_{r}(y_{r}, x- y_{r})u_{r}(y_{r})ds(y_{r}).
	\end{equation}
	Moreover, according to Example 4.10 in the literature, both the single-layer and double-layer potential functions of the Helmholtz equation belong to the previously defined class $\mathcal{P}_{r}$, and thus are also infinitely Gâteaux differentiable. More specifically, here we define the potential operator as:
	$$(P_{\Gamma_{r}}\phi)(x) = \int_{\Gamma_{r}}(\Phi(x,y)\beta - \frac{\Phi(x,y)}{\partial n_{y}}\alpha)dS_{y}.$$
	
	Finally, since the interface correction part of the solution is composed of single-layer and double-layer potentials, we conclude that 
	\begin{equation}
		\begin{split}
			\big\Vert \hat{\mathcal{G}}(\chi_{\Gamma}, \dots) - \hat{\mathcal{G}}(\chi_{\Gamma_r}, \dots) \big\Vert_{2} = \Vert u_{\Gamma} - u_{\Gamma_{r}} \Vert_{L^{2}(\Omega)} = \Vert (P_{\Gamma_{r}} - P_{\Gamma_{0}} )\phi \Vert_{L^{2}(\Omega)} < C\Vert r\Vert_{\infty}
		\end{split}
	\end{equation}
	for some constant $C$. Thus, if we choose $\delta = \frac{\varepsilon}{C}$ our theorem can be proven.
\end{proof}

\begin{theorem}[Continuity of interface under general interface conditions]
	For the operator $\hat{\mathcal{G}}$ defined above, $\Omega$ is a bounded domain in $\mathbb{R}^{n}$(n = 2 or 3), moreover, all interfaces considered here are smooth closed hypersurfaces within \(\Omega\), assume $\alpha(x) \in H^{1/2}(\Gamma)$ and $\beta(x) \in H^{-1/2}(\Gamma)$, it follows that $\hat{\mathcal{G}}$ is continuous with respect to variation $\Gamma$. That is, for any prescribed $\Gamma$, $\Gamma_r$, assume that there exists of a transformation $\tau_{r}$ from $\Gamma$ to $\Gamma_{r}$:
	\begin{equation}
		\tau_{r}: x\mapsto x + r(x),
	\end{equation}
	then for arbitrary $\varepsilon > 0$, there exists $\delta > 0$ such that if $\Vert r\Vert_{\infty} < \delta$, then the following inequality holds:
	{\footnotesize
		\begin{equation}
			\big\Vert \hat{\mathcal{G}}(\phi_{\Omega}, \phi_{\Gamma}, \hat{b}(x), \hat{f}(x), \hat{g}(x), \hat{\alpha}_{\Gamma}(x), \hat{\beta}_{\Gamma}(x)) - 
			\hat{\mathcal{G}}(\phi_{\Omega}, \phi_{\Gamma_{r}}, \hat{b}(x), \hat{f}(x), \hat{g}(x), \hat{\alpha}_{\Gamma_{r}}(x), \hat{\beta}_{\Gamma_{r}}(x)) \big\Vert_{2} < \delta,
		\end{equation}
	}
	here $ \hat{\alpha}_{\Gamma_{r}}(x), \hat{\beta}_{\Gamma_{r}}(x)$ is defined as follows:
	$$\hat{\alpha}_{\Gamma_r} = \widehat{\alpha_{\Gamma}\circ \tau^{-1}_{r}}, \hat{\beta}_{\Gamma_r} = \widehat{\beta_{\Gamma}\circ \tau^{-1}_{r}}.$$
\end{theorem}
\begin{proof}
	Note that the only distinction here from Theorem \ref{conti_c} is that $\alpha$ and $\beta$ are no longer constant functions. Since the theorem in Martin Costabel's work\cite{costabel2012shape} already covers this case, it suffices to define $\hat{\alpha}_{\Gamma_{r}}(x), \hat{\beta}_{\Gamma_{r}}(x)$ following the method outlined in the theorem to meet the required conditions. The remaining steps of the proof are identical to those of Theorem \ref{conti_c}.
\end{proof}

\begin{theorem}[Continuity of domain of Dirichlet zero boundary condition]
	For the operator $\hat{\mathcal{G}}$ defined above, assume $\Omega, \Omega_{r}$ are bounded domains in $\mathbb{R}^{n}$(n = 2 or 3)  with boundary $\partial \Omega, \partial \Omega_{r}$ of class $C^{2}$, and $\Gamma$ is smooth closed hypersurfaces within both $\Omega, \Omega_{r}$. Then if $f(x)$ is a continuous function defined on $ \Omega \cup \Omega_r $, assume there exists a  transformation from $\Omega$ to $\Omega_{r}$ with $\tau_{r}(x) = x + r(x)$, and assume that $\tau_{r}$ is a bijection with $r\in C^{2}$, then for any $\varepsilon > 0$, there exists $\delta > 0$ such that if $\|r\|_{C^{2}} <\delta$, the following holds:
	\begin{equation}
		\big\Vert \hat{\mathcal{G}}(\phi_{\Omega}, \phi_{\Gamma}, \hat{f_{|\Omega}}(x), \dots) - \hat{\mathcal{G}}(\phi_{\Omega_r}, \phi_{\Gamma},  \hat{f_{|\Omega_{r}}}(x), \dots) \big\Vert_{2} < \varepsilon,
	\end{equation}
	here, all other input functions remain unchanged, with $g_{\Omega} = 0, g_{\Omega_{r}} = 0$, and the interface conditions are both set to 0.
\end{theorem}
\begin{proof}
	
	First, we verify some properties of the Helmholtz operator and the domain mapping.
	
	For the Helmholtz operator $A(u) = -\Delta u - k^{2}u$, since the Helmholtz operator is a linear operator, its differentiability with respect to $u$ in the weak sense is trivial. Moreover, since we adopt the same definition of $ r $ as in Example 6.2 of J. Simon's paper \cite{simon1980differentiation}, it follows from the proofs therein that the remaining conditions in Theorem 3.3 are also satisfied.
	
	Next, we further cite Theorem 3.3 from J. Simon's paper, $r \mapsto u(r) \circ (I+r)$ is differentiable. And from this, we can prove that:
	\begin{equation}
		\begin{split}
			&\big\Vert \hat{\mathcal{G}}(\chi_{\Omega}, \hat{f_{|\Omega}}(x),\dots) -  \hat{\mathcal{G}}(\chi_{\Omega_r}, \hat{f_{|\Omega_{r}}}(x),\dots) \big\Vert_{2}\\
			=& \Vert \hat{u}(r) - \hat{u}(0)\Vert_{L^{2}}\\
			=& \int_{\Omega} |u(r) \circ (I+r) - u(0) |^2  dx .
		\end{split}
	\end{equation}
	Based on the properties of the mapping  $r \mapsto u(r) \circ (I+r)$, there exists a constant $M>0$ such that:
	$$\|u(r) \circ (I+r) - u(0)\|_{H^1(\Omega)} \leq M \|\theta\|_{C^2},$$
	Using the continuous embedding properties of Sobolev spaces, we have:
	$$\|u(r) \circ (I+r) - u(0)\|_{L^2(\Omega)} \leq M' \|\theta\|_{C^2}.$$
	Then, the inequality becomes:
	$$\int_{\Omega} |u(r) \circ (I+r) - u(0) |^2  dx = \|u(r) \circ (I+r) - u(0)\|_{L^2(\Omega)}^2 \leq (M')^2 \|\theta\|_{C^2}^2.$$
	
	Therefore, by choosing $\delta = \frac{\sqrt{\varepsilon}}{M'}$, the proof is complete.

\end{proof}

\subsection{Domain encoding analysis}
As we mentioned earlier, we will analyze the error in encoding the domain and interface based on the characteristic function and the signed distance function (SDF). To simplify the problem, we consider the following scenario: assume the domain $\Omega$ is the rectangle $[0, 1]^2$, and the interface $\Gamma$ is the boundary of a closed region within the domain. For the encoding of this domain, we analyze the errors of two encoding methods. Here, we assume that a uniform $n \times n$ grid is used for encoding in both cases, and the extension to higher dimensional cases is similar.

\subsubsection{Error of characteristic function domain encode}

First, the exact representation of the domain $\Omega$ is given by the characteristic function $\chi_{\Omega}$. Based on this, we also define a piecewise constant approximation function $grid_{\Omega}$ through grid point sampling.  

Thus, we have:
\begin{equation}
	E_{RMSE} = \sqrt{ \int_{\Omega} | \chi_{\Omega}(x) - {grid}_{\Omega}(x) |^2 dx}.
\end{equation}
Since the characteristic function can only take the values 0 or 1, the value of the integral here strictly equals the area of the geometric discrepancy, i.e.,
$$\text{MSE} = \text{Area}( \Omega_{true} \oplus \Omega_{grid} ).$$

Considering that $\Omega$ is the boundary of an interior region, the area discrepancy here can only arise from the grid cells intersected by the boundary of $\Omega$, denote as $M_{boundary}$.

First, we have the following estimate:
$$M_{boundary} \le N_{crossings}^{(x)} + N_{crossings}^{(y)} + C_{start},$$
here, $ N_{\text{crossings}}^{(x)} $ represents the number of times the curve crosses the vertical grid lines, $ N_{\text{crossings}}^{(y)} $ the number of times it crosses the horizontal grid lines, and $ C_{\text{start}} $ the cell containing the starting point of the curve. 

Assume the curve is given by the parametric equations $ (x(s), y(s)) $, and the length of $\partial\Omega$ is $L$, we estimate each of these  components separately:
\begin{equation}
	\begin{split}
		&N_{crossings}^{(x)} \le \lceil \frac{\int_0^L |x'(s)| ds}{1/n-1} \rceil = \lceil (n-1) \int_0^L |\cos \theta(s)| ds \rceil,\\
		&N_{crossings}^{(y)} \le \lceil \frac{\int_0^L |y'(s)| ds}{1/n-1} \rceil = \lceil (n-1) \int_0^L |\sin \theta(s)| ds \rceil,\\
		&M_{boundary} \le \lceil (n-1) \int_0^L \left( |\cos \theta(s)| + |\sin \theta(s)| \right) ds \rceil + C,
	\end{split}
\end{equation}
Considering the worst-case scenario for vertices and the maximum value of the kernel $ \left( |\cos \theta| + |\sin \theta| \right) $ is $\sqrt{2} $, we take $ C = 4 $, leading to the following:
$$M_{boundary} \le \lceil \sqrt{2} \cdot L \cdot (n-1) \rceil + 4.$$

From this, we calculate the previous $ E_{\text{RMSE}} :$
\begin{equation}
	E_{RMSE} \leq \sqrt{\frac{\lceil \sqrt{2} \cdot L \cdot (n-1) \rceil + 4}{(n-1)^2}}\propto n^{-\frac{1}{2}}.
\end{equation}

\subsubsection{Error of SDF encode}

We still adopt the same notation as above, the only difference is that we use \(\phi\) to denote the SDF. then
\begin{equation}
	E_{RMSE} = \sqrt{ \int_{\Omega} | \phi_{\Omega}(x) - {grid}_{\Omega}(x) |^2 dx},
\end{equation}
here $grid_\Omega$ is a function based on bilinear interpolation of grid values from SDF.

Similarly, for each grid cell, we first categorize them into two types: smooth cells and singular cells. Cells that do not contain the medial axis are considered smooth cells, while the rest are singular cells. First, we focus on the smooth cell portion.

For a smooth cell, we first define local coordinates $ x, y \in [0, h]$ within the cell. Then, the true Signed Distance Function $\phi_{\Omega}$ is expanded using a second-order Taylor series around the bottom-left corner $(0, 0)$ of the cell, yielding:
$$\phi_{\Omega}(x, y) =\phi_{\Omega}(0,0) + x \phi_x + y \phi_y + xy \phi_{xy}+ \frac{1}{2}x^2 \phi_{xx} + \frac{1}{2}y^2 \phi_{yy} + O(h^3).$$
Meanwhile, the encoded function $grid_{\Omega}$, after bilinear interpolation, takes the following form within each cell:
$$grid_{\Omega}(x, y) = A + Bx + Cy + Dxy.$$
Thus, for the function error within each cell, we have:
\begin{equation}
	\begin{split}
		E(x, y) = \phi_{\Omega}(x, y) - grid_{\Omega}(x, y) &= \frac{1}{2}x^2 \phi_{xx}(\xi) + \frac{1}{2}y^2 \phi_{yy}(\eta) + O(h^3) \\
		&\leq  \frac{h^2}{2} \left( |\phi_{xx}| + |\phi_{yy}| \right) + O(h^3).
	\end{split}
\end{equation}

For the singular cells, if the grid is sufficiently fine, the portion of the medial axis intersected by each grid cell can be approximated as a straight line. Thus, we simplify the interior of each singular cell as follows: assume the local coordinates of the grid cell are $[-\frac{h}{2}, \frac{h}{2}]^2$, and the equation of the medial axis within the cell is $n_x x + n_y y + d = 0$. We introduce normalized coordinates $u, v \in [-1/2, 1/2]$ such that $x = h u$ and $y = h v$.

Thus, within this cell, the SDF function becomes:  
$$\phi_{\Omega}(x, y) = |n_x x + n_y y + d| = h \cdot |n_x u + n_y v + \frac{d}{h}|$$ 
i.e., $\phi_{\Omega}(x, y) = h \cdot \hat{\phi_{\Omega}}(u, v)$. Denoting the bilinear interpolation operator as $\Pi$, the encoded function is:  
$$grid_{\Omega}(x, y) = \Pi(\phi_{\Omega}) = h \cdot \Pi(\hat{\phi_{\Omega}}(u, v)).$$
Therefore, we calculate the integral over a single singular grid cell as follows:
\begin{equation}
	\begin{split}
		\iint_{\Omega_{cell}} (\phi - \tilde{\phi})^2 dx dy& = \iint_{-1/2}^{1/2} (h \hat{\phi} - h \Pi \hat{\phi})^2 \cdot (h^2 du dv)\\
		&= h^4 \cdot \underbrace{\iint_{-1/2}^{1/2} \left( \hat{\phi}_{\delta}(u,v) - \Pi \hat{\phi}_{\delta}(u,v) \right)^2 du dv}_{C(\theta, \delta)},
	\end{split}
\end{equation}
here, $ C(\theta, \delta) $ is an integral value independent of $ h $.

Next, we estimate the overall error across the entire domain. Assuming the medial axis length of the internal interface $\Omega$ is $ L$, based on the estimation from the characteristic function part, the number of grid cells it passes through $ M_{singular} $ satisfies:
$$M_{singular} \le \lceil \sqrt{2} \cdot L \cdot (n-1) \rceil + 4 = M_{max}.$$

Since the error contributed by singular grid cells is greater than that of smooth grid cells, we have the following estimate:
{\footnotesize \begin{equation}
	\begin{split}
		\|E\|_{L^2}^2 & = \int_{\Omega} E^2 dx dy  = \sum_{i \in \text{Singular}} \int_{\Omega_i} E^2 + \sum_{j \in \text{Smooth}} \int_{\Omega_j} E^2 \\
		& \leq  M_{max} \cdot h^{4} \cdot max\{ C(\theta, \delta) \} + ((n-1)^{2} - M_{max}) \cdot max\{(\frac{ |\phi_{xx}| + |\phi_{yy}|}{2})^{2} \} \cdot h^{4} \cdot h^{2}\\
		& = C_{1} (n-1)^{-3} + C_{2} (n-1)^{-4}. 
	\end{split}
\end{equation}}

Finally, we obtain:
$$E_{RMSE} = \sqrt{ \int_{\Omega} | \phi_{\Omega}(x) - {grid}_{\Omega}(x) |^2 dx}\propto n^{-\frac{3}{2}}$$

From this, we can see that when representing a domain with a single interface, the SDF offers higher accuracy. However, in certain special cases—such as when the interface aligns perfectly with a standard grid or when multiple interfaces are present—the characteristic function representation proves to be more advantageous.

\section{Experiments}

Finally, we conducted a series of experiments to validate the proposed method, demonstrating its generality and accuracy. 

\subsection{Change the external region.}

First, we consider the scenario where only the domain $\Omega$ has been changed. Here, we let $\Omega$ be a varying star-shaped domain, while the interface is fixed at $x = 0$. We examine the following Poisson equation interface problem:
\begin{equation}
	\begin{split}
		-\Delta u &= f(x), \quad \text{for}\quad x \in \Omega \backslash \Gamma\\
		u|_{\partial \Omega}& = 0, \\
		[u]|_{\Gamma}& = \alpha, \quad [au']|_{\Gamma} = \beta.
		\label{poisson}
	\end{split}
\end{equation}

We constructed a dataset consisting of $N = 4000$ samples. Each sample $i$ contains a pair of geometric, source and interface parameters, denoted as $(\mathbf{r}^{(i)}, \mathbf{f}^{(i)}, \mathbf{p}^{(i)}, \mathbf{u}^{(i)})$.

\begin{itemize}
	\item \textbf{Geometry}: The domain $\Omega^{(i)}$ is a star-shaped region within $[-0.5, 0.5]^2$. It is parameterized by $M=100$ boundary points sampled in polar coordinates, represented by the radial vector $\mathbf{r}^{(i)} \in \mathbb{R}^{100}$.
	\item \textbf{Source functions}: The source function is a random function generated by a Gaussian process on $[-0.5, 0.5]^2$. It is parameterized using a standard $100 \times 100$ grid, resulting in a matrix $ \mathbf{f}^{(i)} \in \mathbb{R}^{100 \times 100} $.
	\item \textbf{Interface conditions}: The jump conditions are scalar values $\beta^{(i)}$, independently and uniformly sampled from the interval $[0, 0.5]$.
	\item \textbf{Reference solution}: The ground truth solution $u^{(i)}$ is obtained by solving the governing equation using FEniCS. This field serves as the target for calculating the training loss.
\end{itemize}

\textbf{Grid Encoding}: To facilitate processing by the neural operator, we map the raw boundary data onto a uniform $100 \times 100$ Cartesian grid over the domain. We employ characteristic function encoding to represent interface parameters, and finally resulting in an input tensor of size $3\times 100\times100$.

\textbf{Network parameter settings}: The FNO architecture is configured with a depth of $L=5$ layers. Within each Fourier layer, the width of the feature map is set to $d_v = 64$. We apply a low-pass filter in the frequency domain by truncating the spectrum at $k_{max} = 12$ modes.

\textbf{Training Details}: The network parameters are optimized using the Adam algorithm. We set  learning rate to be $\eta = 10^{-3}$ and employ a batch size of $N_b = 5$. The model is trained for 2000 epochs until convergence.

\begin{table}[t]
	\centering
	\begin{tabular}{l c c}
		& Normal derivative jump & function value jump \\ 
		Our method & \textbf{1.41e-2} & \textbf{6.43e-2} \\
		Geo-FNO & 5.14e-2 & 19.94e-2 \\
	\end{tabular}
	\caption{The comparison with the deformation-based method Geo-FNO at randomly generated non-training grid. It can be observed that, our method achieves better performance than Geo-FNO at randomly generated non-training grid points.}
	\label{compare}
\end{table}

\begin{figure}[t]
	\centering
	\includegraphics[width=1.0\textwidth]{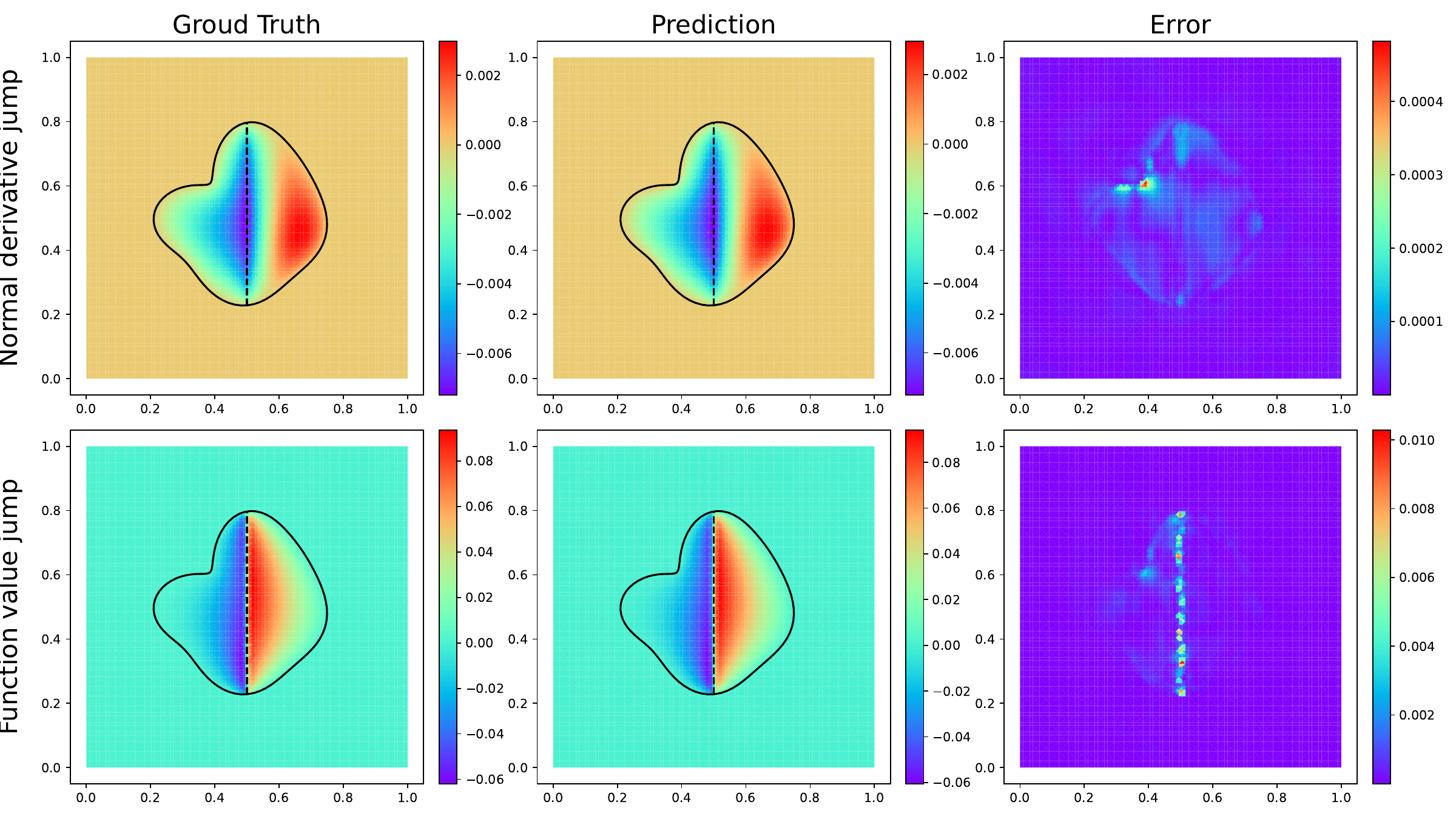}
	\caption{Prediction results when only the external region is varied. The relative $ L^2 $ errors for the two experiments are 1.07\% and 2.69\%, respectively}
	\label{experiment1}
\end{figure}

We also tested cases where the interface function value jump is non-zero. For this discontinuous scenario, aside from fixing the source term $ f $ and setting $ \beta = 0 $ with $ \alpha $ chosen uniformly at random, all other settings remained identical to the previous ones.

We present the results of this experiment in Figure \ref{experiment1}. The first row corresponds to the case with only a jump in the normal derivative, while the second row corresponds to the case with a jump in the function value. The relative $ L^2 $ errors for the two experiments are 1.07\% and 2.69\%, respectively.

Finally, we conducted an additional comparative experiment, comparing our method with the deformation-based Geo-FNO. The experimental results are presented in Table \ref{compare}. Since the original training grids of the two methods are different, we consider the predicted values on another set of randomly generated grid points. For both methods, bilinear interpolation is used to obtain the values at non-training grid points. It can be seen that on the random grid points, our method performs significantly better than Geo-FNO.

\subsection{Change the interface.}

In this experiment, we consider an example where only the internal interface is varied. Here, we set the domain $\Omega$ as $[0,1]^2$, with the internal interface being a star-shaped region that changes within $[0,1]^2$. We still consider the Poisson equation shown in Equation \eqref{poisson}.

We constructed a dataset consisting of $N = 4000$ samples. Each sample $i$ contains a pair of geometric configurations and interface parameters, denoted as $(\mathbf{r}^{(i)}, \mathbf{p}^{(i)}, u^{(i)})$.

\begin{itemize}
	\item \textbf{Geometry}: The interface $\Gamma^{(i)}$ is a star-shaped region within $[0, 1]^2$. It is parameterized by $M=100$ boundary points sampled in polar coordinates, represented by the radial vector $\mathbf{r}^{(i)} \in \mathbb{R}^{100}$.
	\item \textbf{Interface parameters}: The jump conditions consist of two scalar values, $\alpha^{(i)}$ and $\beta^{(i)}$, which are independently and uniformly sampled from the interval $[0, 0.1]$. Collectively, we denote the interface vector as $\mathbf{p}^{(i)} = [\alpha^{(i)}, \beta^{(i)}]$.
	\item \textbf{Reference solution}: The ground truth solution $u^{(i)}$ is obtained by solving the governing equation using FEniCS. 
\end{itemize}

\textbf{Grid Encoding}: To facilitate processing by the neural operator, we map the raw boundary data onto a uniform $100 \times 100$ Cartesian grid over the domain. We tested both encoding methods—the characteristic function and the SDF—and will present a comparison of their results.
This results in an input tensor of size $3 \times 100 \times 100$.

\begin{figure}[t]
	\centering
	\includegraphics[width=1.0\textwidth]{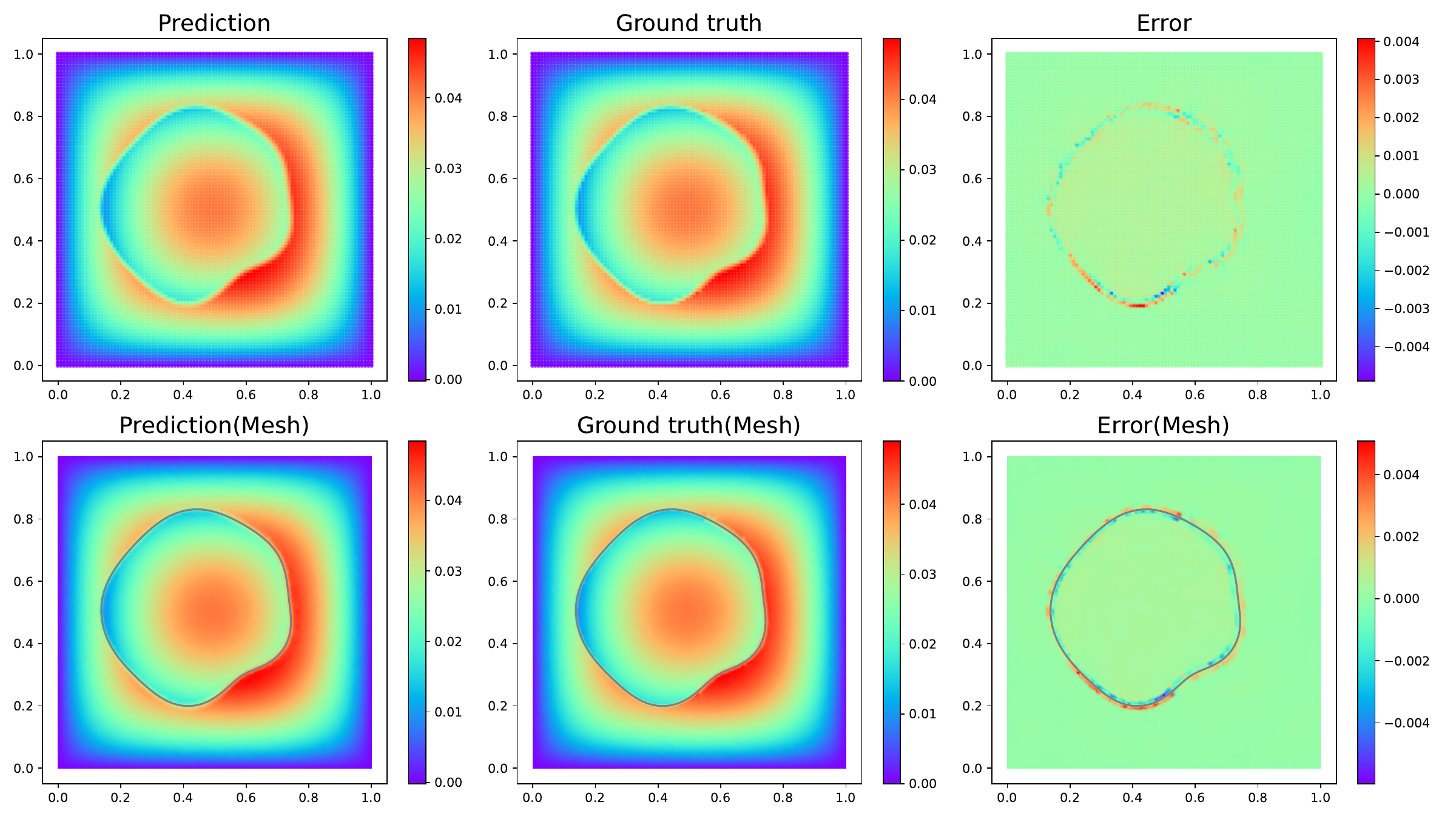}
	\caption{Prediction results for the star-shaped internal interface. The relative $ L^2 $ error for this experiment is 2.20\%.}
	\label{squarepolar}
\end{figure}

In this experiment, the network and training parameters are all identical to those in the previous experiment, so they are not listed separately. We present the prediction results for this experiment in Figure \ref{squarepolar}, where the first row shows the grid point plot and the second row shows the triangular mesh plot. The relative $ L^2 $ error for this experiment is 2.20\%.

\subsection{2D TFPM basis}

In the first two experiments, we separately verified that our method can effectively handle variations in both internal interfaces and external domains. Here, we will conduct a two-dimensional experiment to demonstrate the superiority of the TFPM basis-based approach.

We consider the Helmholtz equation \ref{interfacepde}. For simplicity, we set $a = b = 1 $, $ f(x) = 1 $, $ g(x) = 0 $. The solution domain $\Omega$ is $[0,1]^2$, and the interface $\Gamma$ is a small square with side length 0.4 inside $\Omega$. In this experiment, we prescribe the interface conditions as $\alpha = 0.02$ and $\beta = 0.02$, while allowing the interface geometry $\Gamma$ to vary.

To validate the superiority of the TFPM basis, we constructed a dataset consisting of $N = 2000$ samples. Each sample $i$ contains a pair of geometric configurations and interface parameters, denoted as $(\Omega^{(i)}, \mathbf{p}^{(i)}, u^{(i)})$.

\begin{table}[t]
	\centering
	\begin{tabular}{l c c c}
		& Square($50 \times 50$) & Square($20 \times 20$) & 3D($21 \times 21 \times 21$) \\ 
		FEM\_basis  & 2.65e-2            & 9.81e-2            & 1.26e-2           \\ 
		TFPM\_basis & \textbf{5.23e-3}            & \textbf{4.60e-2}            & \textbf{6.04e-3}           \\ 
	\end{tabular}
	\caption{$L^2$ Error of different basis}
	\label{comparetfpm}
\end{table}

\begin{itemize}
	\item \textbf{Geometry}: The interface $\Gamma^{(i)}$ is small square with side length 0.4 inside $[0, 1]^2$. It is explicitly parameterized by the coordinates of its four vertices, denoted as $\mathbf{v}^{(i)} = \{(x_k, y_k)\}_{k=1}^4$.
	\item \textbf{Interface parameters}: The jump conditions consist of two scalar values, $\alpha^{(i)}$ and $\beta^{(i)}$, which are independently and uniformly sampled constants. Collectively, we denote the interface vector as $\mathbf{p}^{(i)} = [\alpha^{(i)}, \beta^{(i)}]$.
	\item \textbf{Reference solution}: The ground truth solution $u^{(i)}$ is obtained by solving the governing equation using TFPM. 
\end{itemize}

\textbf{Multi-resolution Downsampling}: To demonstrate the resolution-inde\-pendence and superiority of the TFPM basis, we perform systematic downsampling on the raw $100 \times 100$ data. By selecting grid points at fixed strides (subsampling), we generate three additional low-resolution datasets with grid sizes of $50 \times 50$, $20 \times 20$, and $10 \times 10$, respectively.

\begin{figure}[t]
	\centering
	\includegraphics[width=0.9\textwidth]{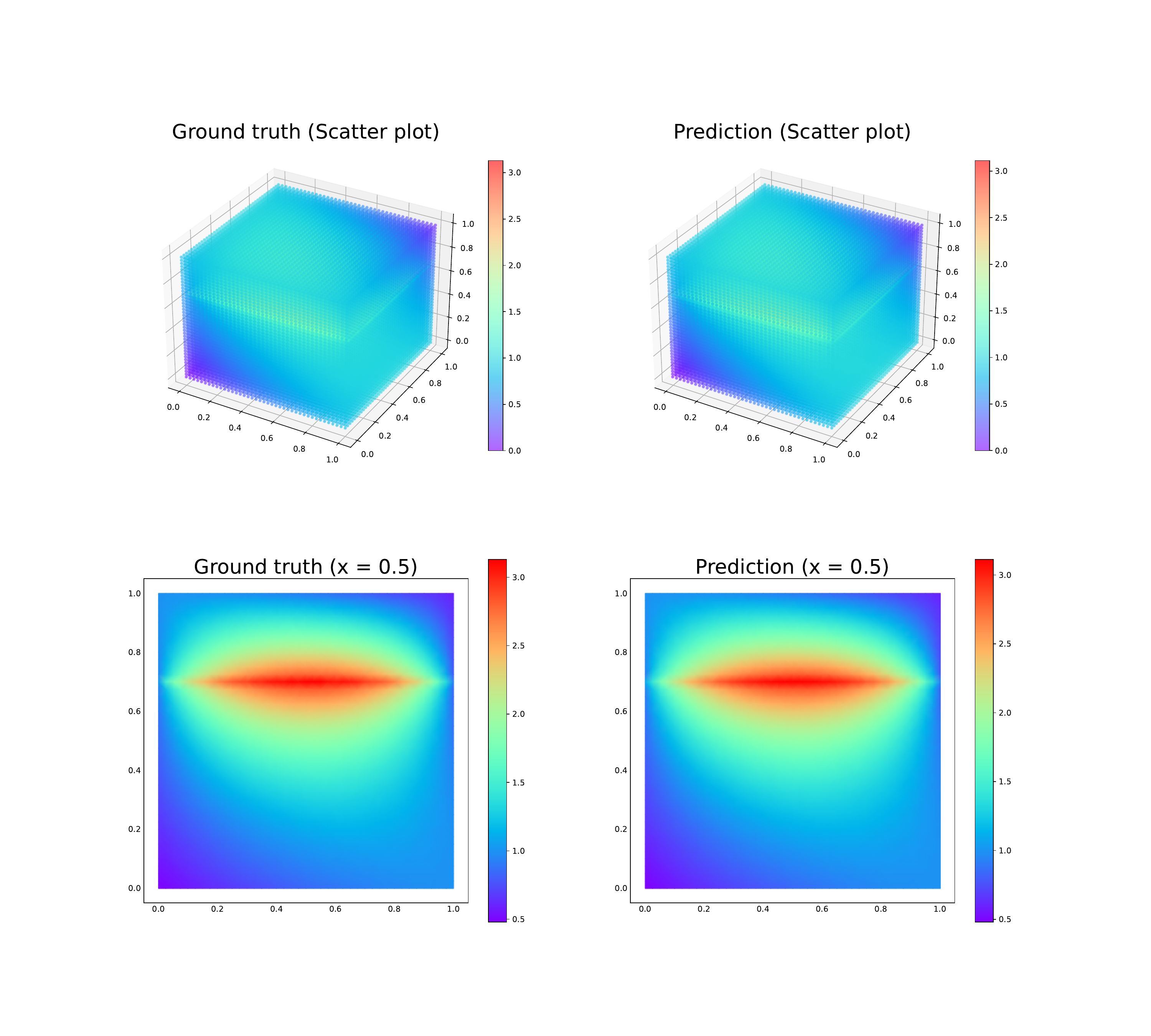}
	\caption{Prediction of 3D Example on TFPM basis. First row: 3D grid point distribution map; Second row: Cross-sectional view at $x = 0.5$.}
	\label{3d}
\end{figure}

\textbf{Grid Encoding}: To facilitate processing by the neural operator, we map the raw boundary data onto a uniform $m \times m$ (where $m < 100$) Cartesian grid over the domain. And we use the signed distance function (SDF) to encode the domain.

\textbf{Output Reconstruction}: In the standard FNO architecture, the spatial resolution remains consistent; an input grid of size $m \times m$ (where $m < 100$) yields a raw prediction on the identical $m \times m$ grid. To evaluate the solutions at the original high resolution, we apply two distinct upsampling strategies: for the finite element basis method, we recover the resolution through standard bilinear interpolation, whereas for the TFPM approach, we explicitly compute the fine grid values within each coarse cell utilizing Equation \eqref{tfpm}.

In this experiment, the network and training parameters are all identical to those in the previous experiment, so they are not listed separately. 

Instead of presenting the results of this experiment separately, we have included them together with the subsequent 3D experiment in Table \ref{comparetfpm}. It can be observed that in the 2D case, replacing the basis with the TFPM basis leads to a significant improvement in prediction performance.

\subsection{3D Example}

Next, to more clearly demonstrate the advantages of the TFPM basis and illustrate that our network can handle high-dimensional cases, we present a three-dimensional experiment. Here, we again consider the Helmholtz equation as in \eqref{tfpm}.

In this experiment, we set $ b = 1.0 $, $ \alpha = 0.2 $, $ \beta = 0 $, $ g(x) = \sin(x + y + z) $, and the interface is defined as $z = \kappa $, where $\kappa$ is a random constant between $ 0.2 $ and $ 0.8 $.

To evaluate the model's capability in handling moving interfaces, we constructed a dataset consisting of $N = 2000$ samples. Out of the total datasets, we randomly partitioned 1000 samples for training and 1000 for testing. Each sample $i$ is characterized by its specific interface configuration, denoted as $(\Omega^{(i)}, u^{(i)})$.
\begin{itemize}
	\item \textbf{Geometry}: The moving interface $\Gamma^{(i)}$ within the domain is explicitly defined by the equation $z = \kappa^{(i)}$. The location of the interface is governed by a single scalar value $\kappa^{(i)}$, which is a random constant independently and uniformly sampled from the interval $[0.2, 0.8]$.
	\item \textbf{Reference solution}: The ground truth solution $u^{(i)}$ is obtained by solving the governing equation using FEniCS. 
\end{itemize}

In this experiment, the network, training parameters and output reconstruction methods are all identical to those in the 2D TFPM basis experiment, so they are not listed separately. We present the experimental results in Figure \ref{3d}. The first row shows the 3D scatter plots, while the second row displays a 2D cross-sectional view. The $L^{2}$ errors for this experiment are listed in the third column of Table \ref{comparetfpm}. In the 3D case, we reduced the geometry term input resolution from $2000 \times 41 \times 41 \times 41$ to $2000 \times 21 \times 21 \times 21$. With a coarser grid, the TFPM-based results show a significant improvement over the baseline.

\subsection{Transport kinetics across interfaces}

Finally, we simulated a problem with a specific physical background.

\begin{figure}[t]
	\centering
	\includegraphics[width=0.8\textwidth]{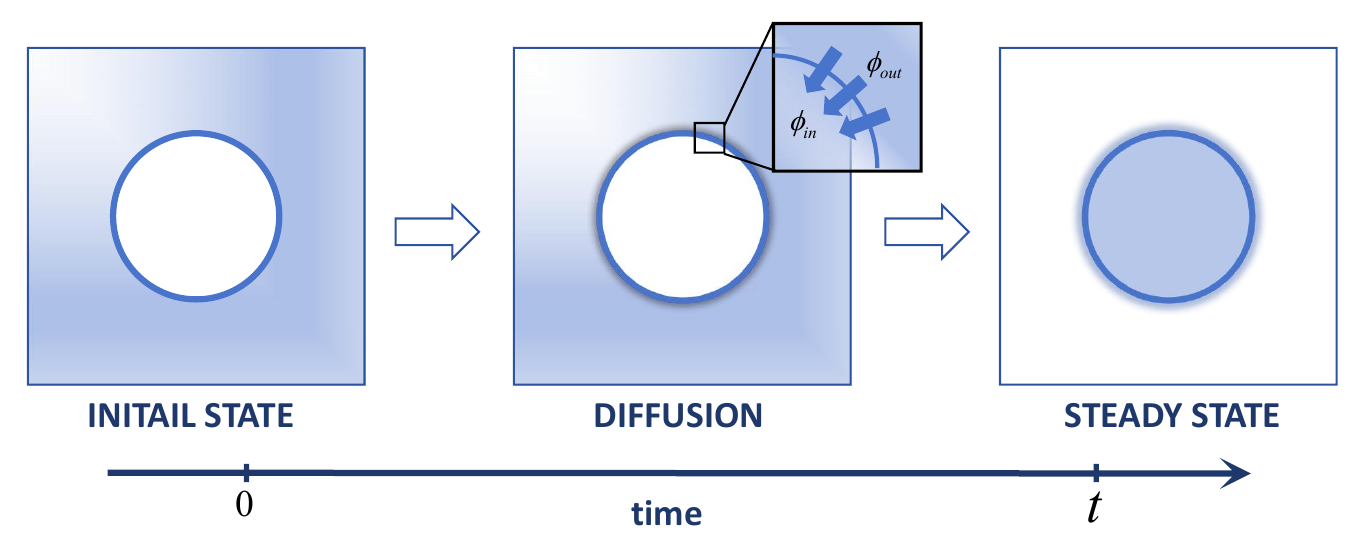}
	\caption{Schematic diagram of transport.}
\end{figure}

We simulated the time-dependent inward diffusion of a solute from a reservoir (matrix) into a solute-free domain (inclusion) across a stationary interface $\Gamma$. The transport kinetics were modeled using the sharp-interface limit framework proposed by Hubatsch et al.\cite{hubatsch2025transport}. The governing dynamics in the bulk sub-domains $\Omega_{in}$ and $\Omega_{out}$ follow the standard diffusion equation $\partial_t \phi = \nabla \cdot (D \nabla \phi)$. A key feature of our setup is the non-equilibrium boundary condition at the interface. Instead of assuming continuous chemical potential (perfect contact), we incorporated an interface resistance parameter $\rho$, which creates a barrier to transport. The mass flux $j$ across the interface is proportional to the deviation from the equilibrium partition, governed by:$$j \cdot n = \frac{1}{\rho} (\phi_{in} - \Gamma^* \phi_{out})$$where $\Gamma^*$ is the equilibrium partition coefficient. The system was initialized with a uniform external concentration $\phi_{out}(t=0) = \phi_0$ and an empty interior $\phi_{in}(t=0) = 0$. We observed the relaxation process as the internal concentration gradually increased towards the equilibrium state $\phi_{in}^{eq} = \Gamma^* \phi_{out}^{eq}$, impeded by the interface resistance $\rho$.

In the experiment, our setup is as follows: the inner region is an ellipse $\Omega$, the outer region is $[0,1]^2 \setminus \Omega$, and the interface is the elliptical boundary $\partial\Omega$. The parameters are set as $\rho = 0.2$, with interior and exterior diffusion coefficients of 0.1 and 0.5, respectively.

We constructed a dataset consisting of $N = 2000$ samples. Each sample $i$ contains a pair of geometric configurations and initial states, denoted as $(\mathbf{g}^{(i)}, u_{init}^{(i)},  u_{t}^{(i)})$.

\begin{itemize}
	\item \textbf{Geometry}: The interface $\Gamma^{(i)}$ is defined as an ellipse centered at $(0.5, 0.5)$, satisfying the equation $(x - 0.5)^2 / (a^{(i)})^2 + (y - 0.5)^2 / (b^{(i)})^2 = 1$. It is parameterized by the semi-major and semi-minor axes, denoted as $\mathbf{g}^{(i)} = [a^{(i)}, b^{(i)}]$, which are randomly sampled.
	\item \textbf{Initial states}: The initial field $u_{init}^{(i)}$ is generated as a piecewise function based on the elliptical interface. The values in the interior region $\Omega_{in}$ are strictly set to 0, while the exterior region $\Omega_{out}$ is initialized with a random function generated by a Gaussian process.
	\item \textbf{Reference solution}: The ground truth solution $u_{t}^{(i)}$ at $t = 0.1$ is obtained by solving the governing equation using FEniCS and time step of 0.01. 
\end{itemize}

\begin{figure}[t]
	\centering
	\includegraphics[width=0.8\textwidth]{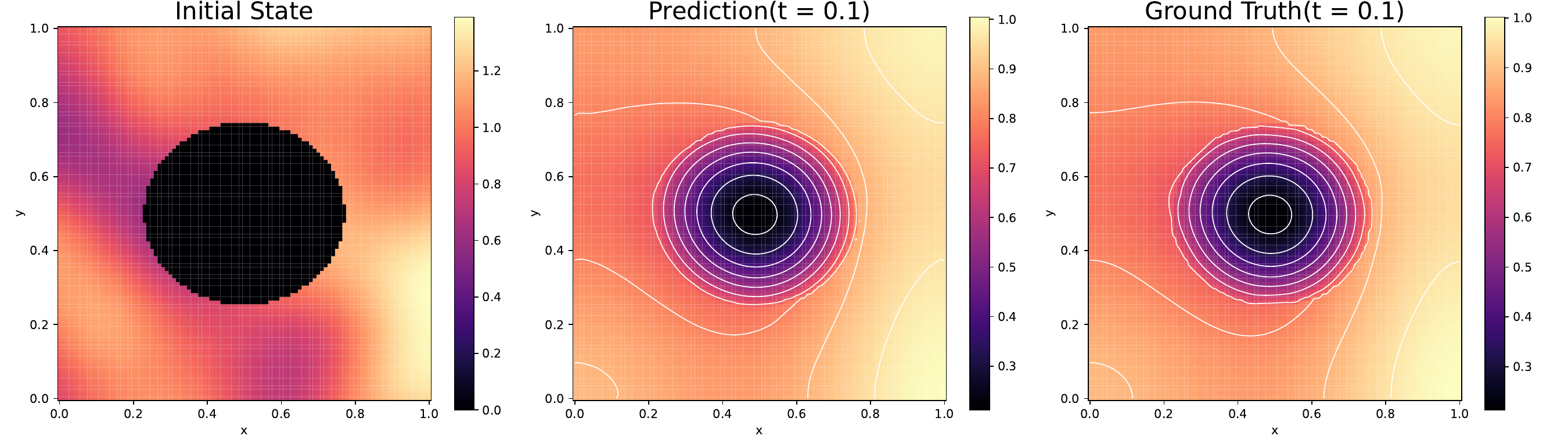}
	\caption{Prediction of transport kinetics of $L^2$ Error 0.31\%. }
	\label{transport}
\end{figure}

In this experiment, the network and training parameters are all identical to those in the previous experiment, so they are not listed separately. We present the prediction results in Figure \ref{transport}. In the comparison plot, we have drawn white contour lines, from which it can be seen that the model achieves good predictive performance, with an average $ L^2 $ error of 0.31\%.

\section{Conclusions}

In this work, we have presented a novel neural operator framework for solving linear interface problems on variable domains. We propose a general extension-based framework for linear PDE interface problems with varying domains, integrated with a TFPM basis-based approach to minimize GPU memory usage, particularly in high dimensions. We establish a rigorous theoretical foundation by proving the continuity of the Helmholtz operator under geometric variations and deriving error estimates for our encoding strategies. The resulting method is broadly applicable, handling complex interface geometries with minimal restrictions.

The main contributions of this paper are as follows:
\begin{itemize}
	\item We propose an extension-based framework for solving interface problems involving variations in both the domain and interfaces, along with encoding strategies suitable for general problems with multiple interfaces.
	\item We integrate the baseline method with the Tailored Finite Point Method (TFPM), introducing a TFPM basis-based approach for varying-domain interface problems. This approach effectively reduces GPU memory consumption for input data and is particularly advantageous in high-dimensional problems. 
	\item Theoretically, we conduct a series of theoretical analyses, including proving the continuity of the Helmholtz operator with respect to variations in the domain and interfaces under specific conditions, as well as providing error estimates for two domain encoding strategies.
\end{itemize}

While the proposed method demonstrates robust performance in different scenarios, extending it to higher dimensions still presents a significant challenge. A critical bottleneck is the reliance on standard grid-based encoding, which suffers from the curse of dimensionality; as the dimension increases, the size of the encoding tensor grows exponentially, imposing severe computational and memory overheads. Consequently, our future work will focus on developing more efficient, low-rank, or sparse encoding schemes. Our goal is to represent high-dimensional functions and domains with significantly less data, thereby ensuring both scalability and superior performance in high-dimensional applications.

\section*{Acknowledgments}

This research is supported by National Natural Science Foundation of China (Grant Nos. 12025104, Nos. 62106103) and the basic research project (ILF240021A25).

\end{document}